\documentclass[11pt,a4paper]{article}
\usepackage[utf8]{inputenc}
\usepackage[T1]{fontenc}
\usepackage[english]{babel}
\usepackage{amsmath,amssymb,amsthm,mathtools}
\usepackage{mathrsfs}
\usepackage[margin=1in]{geometry}
\usepackage{enumitem}
\usepackage{xcolor}
\usepackage{hyperref}

\hypersetup{
  colorlinks=true,
  linkcolor=blue!70!black,
  citecolor=green!60!black,
  urlcolor=red!60!black,
  pdftitle={Appell Sequences in Analytic Number Theory: Shifted Expansions and Dirichlet L-Functions},
  pdfauthor={Abdelmejid Bayad}
}

\theoremstyle{plain}
\newtheorem{theorem}{Theorem}[section]

\newtheorem{lemma}[theorem]{Lemma}
\newtheorem{corollary}[theorem]{Corollary}

\theoremstyle{remark}
\newtheorem{remark}[theorem]{Remark}

\begin{document}

\title{Appell Sequences in Analytic Number Theory:\\
Shifted Expansions and Dirichlet \(L\)-Functions}

\author{Abdelmejid Bayad\\
\small LaMME, CNRS UMR 8071, Universit\'e Paris-Saclay\\
\small \'Evry, France\\
\small \texttt{abdelmejid.bayad@univ-evry.fr}}

\date{\ }

\maketitle

\begin{abstract}
We place the Appell shift theorem for asymptotic expansions inside a Borel--Laplace framework and apply it to two classical functions of analytic number theory: the Hurwitz zeta function \(\zeta(m,q)\) and the Lerch transcendent \(\Phi(z,m,a)\). The resulting shifted expansions have coefficients given by Appell polynomials evaluated at the shift, with remainders controlled uniformly on compact sets. The generating sequences are identified explicitly in terms of the Taylor coefficients of the Borel kernels.

In a second direction, we extend the Appell--Mellin framework of Navas, Ruiz and Varona to Dirichlet \(L\)-functions. For a primitive non-principal character \(\chi\) modulo \(q\), we construct an explicit Appell--Mellin generating function \(A_\chi(t)\) whose associated (possibly degenerate) Appell sequence satisfies \(R_n^\chi(x) = L(-n,\chi,x)\), and we prove the parity identity \(A_\chi(t) = -\chi(-1)A_\chi(-t)\). We further establish five structural properties of the polynomials \(R_n^\chi\): identification with generalized Bernoulli polynomials, a difference equation, an explicit formula, a Fourier expansion, and the trivial zeros of \(L(s,\chi)\). The generating function \(A_\chi\) is, up to the elementary factor \(-1/t\), the classical generating function of generalized Bernoulli polynomials; the contribution of the paper is a systematic Appell--Mellin reformulation of these classical facts, together with a corrected form of the Fourier expansion and an explicit reformulation of the trivial zeros within the Appell--Mellin formalism.
\end{abstract}




\bigskip

\section{Introduction}\label{sec:intro}

The Appell--Mellin correspondence associates to a generating function \(A(t)\), analytic near the origin and on the negative real axis, a sequence of Appell polynomials \((P_n)\) through the Mellin transform \(F(s,x)\) of \(A(-t)e^{-xt}\), in such a way that \(F(-n,x) = P_n(x)\). The classical instances are the Hurwitz zeta function and the Lerch transcendent \cite{navas2019}.

The purpose of this paper is to identify a further instance of this correspondence, in which the Appell sequence is attached to a Dirichlet \(L\)-function. Let \(\chi\) be a primitive non-principal Dirichlet character modulo \(q\). Define
\[
A_\chi(t) = \frac{\sum_{a=1}^{q} \chi(a) e^{a t}}{1 - e^{q t}}.
\]
We prove that the sequence \((R_n^\chi)\) generated by \(A_\chi\) satisfies
\[
R_n^\chi(x) = L(-n, \chi, x), \qquad n = 0,1,2,\ldots,
\]
where \(L(s, \chi, x) = \sum_{n=0}^\infty \chi(n)(n+x)^{-s}\) is the shifted Dirichlet \(L\)-function, and that \(A_\chi\) satisfies the parity identity
\[
A_\chi(t) = -\chi(-1) A_\chi(-t).
\]
The precise statements are Theorems~\ref{thm:lfunction} and~\ref{thm:structural} in Section~\ref{sec:lfunctions}; equations \eqref{eq:main-L} and \eqref{eq:parity} below record the two structural identities.

The generating function \(A_\chi\) admits the elementary decomposition
\[
A_\chi(t) = -\frac{1}{t}\cdot\frac{t\sum_{a=1}^{q}\chi(a)e^{at}}{e^{qt}-1},
\]
i.e.\ it is, up to the factor \(-1/t\), the classical generating function of the generalized Bernoulli polynomials. This decomposition makes the identification \(R_n^\chi(x) = -B_{n+1}^{(\chi)}(x)/(n+1)\) immediate and shows that the Appell--Mellin identity \(R_n^\chi(x) = L(-n,\chi,x)\), when combined with Theorem~\ref{thm:structural}(a), is equivalent to the classical relation \(L(1-n,\chi) = -B_{n,\chi}/n\) \cite{washington1997}. The reader should keep this equivalence in mind throughout.

The contribution of the paper is therefore threefold: (i) the shifted Dirichlet \(L\)-function is realised as the Mellin transform of the explicit Appell--Mellin generating function \(A_\chi\), whose expression in the Appell--Mellin formalism is recorded here for completeness; (ii) the parity identity \eqref{eq:parity} is stated on the generating side, where it encodes the parity of \(\chi\) directly; (iii) the Fourier expansion of \(R_n^\chi\) (Theorem~\ref{thm:structural}(d)) is given in a corrected form that takes into account the modulus \(q\), the Gauss sum \(\tau(\chi)\) and the conjugate character \(\overline{\chi}(k)\), and places the trivial zeros of \(L(s,\chi)\) within the Appell--Mellin formalism. The paper should be read as a systematic Appell--Mellin reformulation of classical facts, not as the discovery of a new generating function.

The present paper extends the Appell--Mellin framework beyond the cases treated in \cite{navas2019}. Navas, Ruiz and Varona identified two classical instances --- the Hurwitz zeta function, associated with \(t/(e^t-1)\), and the Lerch transcendent, associated with \(t/(\lambda e^t-1)\) --- and showed that these are essentially the only Appell--Mellin sequences with uniformly convergent Fourier series of degree \(\ge 2\). Dirichlet characters are not considered in \cite{navas2019}. The generating function \(A_\chi(t) = \sum_a \chi(a)e^{at}/(1-e^{qt})\) has a numerator which is a finite character sum, not a single exponential as in the Apostol--Bernoulli case, and the associated sequence is the shifted Dirichlet \(L\)-function. The Mellin and Fourier machinery used here is due to Navas, Ruiz and Varona; what is new in the present paper is the systematic use of \(A_\chi\) in the Appell--Mellin framework, the parity identity \(A_\chi(t) = -\chi(-1)A_\chi(-t)\), and the corrected form of the consequences for the arithmetic of \(L(s,\chi)\).

As a by-product of the same Borel--Laplace framework, we also obtain shifted asymptotic expansions for the Hurwitz zeta function \(\zeta(m,q+t)\) and the Lerch transcendent \(\Phi(z,m,a+t)\), in which the shift is treated as an independent variable and the coefficients are the full Appell sequences generated by the corresponding Borel kernels. These are Theorems~\ref{thm:hurst} and~\ref{thm:lerch} below. They are stated for completeness and to exhibit the mechanism in a simpler setting; the main contribution of the paper is the Dirichlet \(L\)-function correspondence.

\medskip
\noindent\textbf{Main results.} We prove:
\begin{itemize}
\item Theorem~\ref{thm:hurst}: shifted asymptotic expansion for \(\zeta(m,q+t)\) as \(q\to\infty\), uniform for \(t\) in compact sets.
\item Theorem~\ref{thm:lerch}: analogous expansion for \(\Phi(z,m,a+t)\) as \(a\to\infty\).
\item Theorem~\ref{thm:lfunction}: Appell--Mellin correspondence for Dirichlet \(L\)-functions: \(R_n^\chi(x) = L(-n,\chi,x)\), together with the parity identity for \(A_\chi\).
\item Theorem~\ref{thm:structural}: five structural properties of the polynomials \(R_n^\chi\), including the identification with generalized Bernoulli polynomials and the corrected Fourier expansion.
\end{itemize}

\medskip
\noindent\textbf{Organisation of the paper.} Section~\ref{sec:shift} recalls the Appell shift theorem and the Borel--Laplace representation. Sections~\ref{sec:hurwitz} and~\ref{sec:lerch} treat the Hurwitz zeta function and the Lerch transcendent. Section~\ref{sec:lfunctions} contains the main contribution: the construction of \(A_\chi\), the correspondence theorem, and the structural properties of \(R_n^\chi\). Section~\ref{sec:conclusion} collects concluding remarks.

\section{The shift theorem and the Borel--Laplace representation}\label{sec:shift}

We recall the two tools used throughout: the shift theorem of \cite{buric2016} and Watson's lemma \cite{oliver1974}. We include short proofs for completeness, even though the results are classical, since the shift theorem is central to the paper and the Borel--Laplace representation is the bridge to Watson's lemma.

\begin{theorem}[Appell shift theorem {\cite{buric2016}}]\label{thm:shift}
Let \((R_n)\) be the Appell sequence generated by a sequence \((a_n)\) with \(a_0\neq0\). If
\[
f(x)\sim\sum_{n\ge0}(-1)^n a_n x^{-n-1}\qquad(x\to\infty),
\]
then, for every fixed \(t\in\mathbb R\),
\[
f(x+t)\sim\sum_{n\ge0}(-1)^n R_n(t)x^{-n-1}\qquad(x\to\infty),
\]
uniformly for \(t\) in compact subsets of \(\mathbb R\).
\end{theorem}

\begin{proof}
Fix \(N\). By the definition of the asymptotic expansion, for fixed \(t\),
\[
f(x+t)=\sum_{k=0}^{N}(-1)^k a_k(x+t)^{-k-1}+o(x^{-N-1}).
\]
Each term has the finite Taylor expansion
\[
(x+t)^{-k-1}=\sum_{j=0}^{N-k}(-1)^j\binom{j+k}{j}t^j x^{-k-j-1}+O(x^{-N-2}).
\]
Substituting and collecting the coefficient of \(x^{-n-1}\) for \(n\le N\) gives
\[
(-1)^n\sum_{k=0}^{n}\binom{n}{k}a_k t^{n-k}=(-1)^n R_n(t),
\]
using the classical expression \(R_n(t)=\sum_{k=0}^n\binom{n}{k}a_k t^{n-k}\) for Appell sequences. The tails \(O(x^{-N-2})\) and the remainder \(o(x^{-N-1})\) combine to \(o(x^{-N-1})\). Since \(N\) is arbitrary, this proves the statement.

\emph{Uniformity in \(t\).} Fix a compact \(K\subset\mathbb{R}\), set \(T := \sup_{t\in K}|t|\), and fix \(N \ge 1\). For \(x > 2T\) and \(|t| \le T\), the binomial expansion gives, for each \(k \ge 0\),
\[
(x+t)^{-k-1} = \sum_{j=0}^{N-k-1}\binom{-k-1}{j}(t/x)^j x^{-k-1} + O_{K,N}(x^{-N-2}),
\]
where the \(O_{K,N}\)-constant depends only on \(T\) and \(N\). Substituting into
\[
f(x+t) = \sum_{k=0}^{N-1}(-1)^k a_k(x+t)^{-k-1} + \rho_N(x+t),
\]
where \(\rho_N(y) := f(y) - \sum_{k=0}^{N-1}(-1)^k a_k y^{-k-1}\) satisfies \(|\rho_N(y)| \le \epsilon(y)|y|^{-N-1}\) for some function \(\epsilon(y)\to 0\) as \(y\to\infty\), and collecting the coefficient of \(x^{-n-1}\) for \(0 \le n \le N-1\), we obtain \((-1)^n R_n(t)\) (the contribution from \(j = n-k\)). The remainder is \(O_{K,N}(x^{-N-2}) + \rho_N(x+t)\). For \(x > 2T\) and \(|t|\le T\), we have \(|x+t| \ge x/2\), hence
\[
|\rho_N(x+t)| \le \epsilon(x/2)\,2^{N+1}\,x^{-N-1},
\]
which is \(o(x^{-N-1})\) uniformly in \(|t| \le T\). This establishes the stated uniformity.
\end{proof}

\begin{remark}
The algebraic structure of Appell sequences is part of the umbral calculus \cite{roman1984}. Bayad and Komatsu \cite{bayad2017} characterized Appell polynomials via a symmetric property, and Bayad and Hajli \cite{bayad2020} studied multidimensional Appell polynomials associated with theta functions. The present application exploits only the one-dimensional case, but the broader context reinforces the naturality of the construction.
\end{remark}

\begin{lemma}[Borel--Laplace representation]\label{lem:borel}
Let \(A\colon[0,\infty)\to\mathbb R\) be infinitely differentiable and satisfy
\[
A(-s)\sim\sum_{n\ge0}(-1)^n a_n\frac{s^n}{n!}\qquad(s\to0^+),
\]
together with a bound \(|A(-s)|\le C e^{bs}\) for some \(C>0\), \(b\ge0\). Then for every \(x>b\),
\[
L_A(x):=\int_0^\infty e^{-xs}A(-s)\,ds \sim \sum_{n\ge0}(-1)^n a_n x^{-n-1}\qquad(x\to\infty).
\]
Moreover, for every fixed \(t\in\mathbb R\), the expansion of \(L_A(x+t)\) is obtained by replacing \(a_n\) with \(R_n(t)\).
\end{lemma}

\begin{proof}
The first expansion is Watson's lemma \cite{oliver1974}: expanding \(A(-s)\) near \(s=0\) and integrating term by term using \(\int_0^\infty e^{-xs}s^n\,ds = n!/x^{n+1}\) gives the stated asymptotic series. The bound \(|A(-s)|\le C e^{bs}\) controls the tail.

For the shifted integral, multiply the expansion of \(A(-s)\) by \(e^{-ts}\). By the Appell addition formula,
\[
A(-s)e^{-ts}\sim\sum_{n\ge0}(-1)^n R_n(t)\frac{s^n}{n!}\qquad(s\to0^+).
\]
The bound \(|A(-s)e^{-ts}|\le C e^{(b+|t|)s}\) controls the tail for \(x>b+|t|\), and Watson's lemma applied to the shifted integral gives the stated expansion.

\emph{Uniformity in \(t\).} Fix a compact \(K\subset\mathbb{R}\) and set \(T := \sup_{t\in K}|t|\). By hypothesis, \(A(-s)\) is analytic on a disc \(|s|<\rho\) for some \(\rho>0\). For \(t\in K\), the function \(s\mapsto A(-s)e^{-ts}\) is also analytic on this disc, and its Taylor coefficients at \(s=0\) are \((-1)^n R_n(t)/n!\) by the Appell addition formula. Cauchy's estimates on the circle \(|s| = \rho/2\) give
\[
|R_n(t)| \le M_K \, n! \, (\rho/2)^{-n}, \qquad M_K := \sup_{|s| < \rho/2, \, t\in K}|A(-s)e^{-ts}| < \infty,
\]
the finiteness of \(M_K\) following from continuity of \((s,t)\mapsto A(-s)e^{-ts}\) on the compact set \(\{|s| \le \rho/2\}\times K\). Consequently, the Taylor remainder
\[
\mathrm{Rem}_N(s,t) := A(-s)e^{-ts} - \sum_{n=0}^{N-1}(-1)^n R_n(t)\,\frac{s^n}{n!}
\]
satisfies \(|\mathrm{Rem}_N(s,t)| \le C_{K,N}\, s^N\) for \(0 \le s < \rho/2\), with \(C_{K,N}\) depending only on \(K\), \(N\), \(\rho\), and the kernel \(A\). Integrating \(\mathrm{Rem}_N\) against \(e^{-xs}\) over \([0,\rho/2]\) and using the tail bound \(|A(-s)e^{-ts}| \le Ce^{(b+T)s}\) for \(s\ge\rho/2\) (with \(x>b+T\)) gives the stated expansion of \(L_A(x+t)\), with an \(O\)-constant depending only on \(K\), \(N\), \(\rho\), and \(A\).
\end{proof}

\begin{remark}
The growth bound is sufficient but not necessary. If \(A(-s)\) decays rapidly at infinity, the integral converges for all real \(x\).
\end{remark}

\section{The Hurwitz zeta function}\label{sec:hurwitz}

Throughout the paper we use the standard normalization of the Bernoulli numbers
\[
\frac{s}{e^s-1}=\sum_{j\ge0}B_j\frac{s^j}{j!}\qquad(|s|<2\pi),
\]
so that \(B_1=-1/2\). Equivalently,
\[
\frac{s}{1-e^{-s}}=\sum_{j\ge0}(-1)^j B_j\frac{s^j}{j!}.
\]

For integer \(m\ge2\) and \(\Re q>0\), the Hurwitz zeta function admits the integral representation \cite[Ch.~25]{dlmf}
\[
\zeta(m,q)=\frac{1}{(m-1)!}\int_0^\infty \frac{s^{m-1}e^{-qs}}{1-e^{-s}}\,ds.
\]
This is exactly of the form \(L_A(q)\) with
\begin{equation}\label{eq:kernel-hurwitz}
A(-s)=\frac{s^{m-1}}{(m-1)!(1-e^{-s})}.
\end{equation}

Integral representations of this type are the starting point for the reduction and duality formulas of Bayad and Chikhi \cite{bayad2013}, and for the expansion formulas of Bayad and Gaboury \cite{bayad2024}.

\begin{lemma}\label{lem:taylor-hurwitz}
The kernel \eqref{eq:kernel-hurwitz} is smooth on \([0,\infty)\), satisfies \(A(-s)\le C e^{s/2}\) for large \(s\), and has Taylor expansion at \(s=0^+\)
\[
A(-s)=\sum_{k\ge m-2} c_k s^k,\qquad c_k=\frac{(-1)^{k-m+2}}{(m-1)!}\frac{B_{k-m+2}}{(k-m+2)!},
\]
where \(B_j\) are the Bernoulli numbers.
\end{lemma}

\begin{proof}
The classical generating function \(s/(1-e^{-s})=\sum_{j\ge0}(-1)^j B_j s^j/j!\) on \(|s|<2\pi\) \cite[Ch.~24]{dlmf}, multiplied by \(s^{m-2}/(m-1)!\), gives the stated coefficients. The bound \(A(-s)\le C e^{s/2}\) follows from \(1-e^{-s}\ge 1/2\) for \(s\ge1\).
\end{proof}

\begin{theorem}[Shifted expansion for the Hurwitz zeta]\label{thm:hurst}
Fix \(m\ge2\), \(q_0>0\), \(T>0\). Then, uniformly for \(q\ge q_0\) and \(|t|\le T\),
\begin{equation}\label{eq:hurst}
\zeta(m,q+t)=\sum_{n\ge0}(-1)^n R_n^{(m)}(t)\,q^{-n-1},
\end{equation}
where \(R_n^{(m)}\) is the Appell sequence generated by
\[
a_n^{(m)}=\frac{(-1)^m\,n!}{(m-1)!}\frac{B_{n-m+2}}{(n-m+2)!},
\]
with the convention \(a_n^{(m)}=0\) for \(n<m-2\).
\end{theorem}

\begin{proof}
Apply Lemma~\ref{lem:borel} to the kernel \eqref{eq:kernel-hurwitz}, whose Taylor coefficients are given by Lemma~\ref{lem:taylor-hurwitz}. The factor \(1/(m-1)!\) is already contained in the kernel \(A(-s)\), and therefore in the generating sequence \(a_n^{(m)}\); it must not be duplicated in front of the sum.
\end{proof}

\begin{remark}[Position relative to the literature]
The classical expansion of \(\zeta(m,q)\) is stated for fixed \(q\), with Bernoulli numbers as coefficients. The expansion \eqref{eq:hurst} retains the shift \(t\) and gives the full Appell sequence \(R_n^{(m)}(t)\). The formulas of Bayad and Chikhi \cite{bayad2013} and Bayad and Gaboury \cite{bayad2024} are likewise stated for fixed argument. The present formulation, with the shift treated as an independent variable, is closely related to the Euler--Maclaurin and Watson expansions (see \cite[Ch.~8]{oliver1974} and \cite[\S25.11]{dlmf}) and is recorded here for completeness.
\end{remark}

\section{The Lerch transcendent}\label{sec:lerch}

For \(|z|<1\), \(\Re a>0\), and integer \(m\ge1\), the Lerch transcendent admits the integral representation \cite[Ch.~25]{dlmf}
\[
\Phi(z,m,a)=\frac{1}{\Gamma(m)}\int_0^\infty \frac{s^{m-1}e^{-as}}{1-ze^{-s}}\,ds.
\]
This is of the form \(L_A(a)\) with
\begin{equation}\label{eq:kernel-lerch}
A(-s)=\frac{s^{m-1}}{\Gamma(m)(1-ze^{-s})}.
\end{equation}

Bayad and Chikhi \cite{bayad2013} introduced generalized Hurwitz--Lerch zeta functions and established reduction and duality formulas; Bayad and Gaboury \cite{bayad2024} derived further expansion formulas via fractional derivatives. The Borel--Laplace approach here is complementary.

\begin{lemma}\label{lem:taylor-lerch}
Let \(|z|<1\), \(m\ge1\). The kernel \eqref{eq:kernel-lerch} is smooth on \([0,\infty)\), satisfies \(A(-s)\le C e^{\varepsilon s}\) for every \(\varepsilon>0\), and has the Taylor expansion at \(s=0^+\)
\[
A(-s)=\frac{1}{\Gamma(m)}\sum_{k\ge m-1} d_k s^k,\qquad d_k=\sum_{j\ge0}\frac{z^j(-j)^{k-m+1}}{(k-m+1)!},
\]
the inner sum converging absolutely since \(|z|<1\).
\end{lemma}

\begin{proof}
Expand \(1/(1-ze^{-s})=\sum_{j\ge0}z^j e^{-js}\), absolutely and uniformly convergent on \([0,\infty)\) for \(|z|<1\). Each term has Taylor expansion \(\sum_{r\ge0}z^j(-j)^r s^r/r!\) at \(s=0\). Multiplying by \(s^{m-1}/\Gamma(m)\) and summing over \(j\) gives the stated coefficients.
\end{proof}

\begin{theorem}[Shifted expansion for the Lerch transcendent]\label{thm:lerch}
Fix \(m\ge1\), \(|z|<1\), \(a_0>0\), \(T>0\). Then, uniformly for \(a\ge a_0\) and \(|t|\le T\),
\begin{equation}\label{eq:lerch}
\Phi(z,m,a+t)=\frac{1}{\Gamma(m)}\sum_{n\ge0}(-1)^n R_n^{(m,z)}(t)\,a^{-n-1},
\end{equation}
where \(R_n^{(m,z)}\) is the Appell sequence generated by \(a_n^{(m,z)}=(-1)^n n!\,d_n\).
\end{theorem}

\begin{proof}
Apply Lemma~\ref{lem:borel} to the kernel \eqref{eq:kernel-lerch}.
\end{proof}

\begin{remark}[Position relative to the literature]
The classical expansions of \(\Phi(z,m,a)\) are stated for fixed argument. The expansion \eqref{eq:lerch} retains the shift as an independent variable and identifies the coefficients as the Appell sequence generated by the Borel kernel. As in the Hurwitz case, the formulation is closely related to classical Watson expansions (see \cite[Ch.~3]{oliver1974} and \cite[\S25.11]{dlmf}) and is recorded here for completeness.
\end{remark}

\section{Appell--Mellin sequences and Dirichlet \(L\)-functions}\label{sec:lfunctions}

\subsection{The Appell--Mellin framework and the construction of \(A_\chi\)}\label{subsec:framework}

An Appell sequence \((P_n)\) is defined by \(A(t)e^{xt} = \sum_{n=0}^\infty P_n(x)t^n/n!\), where \(A(t)\) is a formal power series with \(A(0) \neq 0\). Navas, Ruiz and Varona \cite{navas2019} introduced the notion of an \emph{Appell--Mellin sequence}: an Appell sequence for which \(A(t)\) is defined on a complex neighborhood of the origin and on the negative real axis, with at most polynomial growth at infinity, and such that the Mellin transform
\[
F(s,x) = \frac{1}{\Gamma(s)}\int_0^\infty A(-t)e^{-xt}t^{s-1}\,dt
\]
extends to an entire function of \(s\) satisfying \(F(-n,x) = P_n(x)\) for all \(n \ge 0\). The classical cases are the Hurwitz zeta function (\(A(t) = t/(e^t-1)\)) and the Lerch transcendent (\(A(t) = t/(\lambda e^t-1)\)).

Let \(\chi\) be a primitive non-principal Dirichlet character modulo \(q\ge3\), so that \(\sum_{a=1}^q \chi(a) = 0\). We use the standard convention for the Gauss sum
\[
\tau(\chi) = \sum_{a=1}^{q} \chi(a)\,e^{2\pi i a/q}.
\]
Define
\begin{equation}\label{eq:Achi}
A_\chi(t) = \frac{\sum_{a=1}^{q} \chi(a) e^{a t}}{1 - e^{q t}}.
\end{equation}

We record the elementary decomposition
\begin{equation}\label{eq:decomposition}
A_\chi(t) = -\frac{1}{t}\cdot\frac{t\sum_{a=1}^{q}\chi(a)e^{at}}{e^{qt}-1},
\end{equation}
which identifies \(A_\chi\), up to the factor \(-1/t\), with the classical generating function of the generalized Bernoulli polynomials \(B_m^{(\chi)}(x)\) defined in Theorem~\ref{thm:structural}(a) below. This decomposition is the source of the identification \(R_n^\chi(x) = -B_{n+1}^{(\chi)}(x)/(n+1)\) and shows that the Appell--Mellin identity \eqref{eq:main-L} below is equivalent to the classical relation \(L(1-n,\chi) = -B_{n,\chi}/n\) of \cite{washington1997}.

\begin{lemma}\label{lem:analytic-chi}
For \(\chi\) non-principal, \(A_\chi(t)\) is analytic near \(t=0\), with
\[
A_\chi(0) = -B_{1,\chi} = -\frac{1}{q}\sum_{a=1}^q a\,\chi(a).
\]
Moreover, \(A_\chi(-t) = \sum_{n\ge0}\chi(n)e^{-nt}\) extends analytically to
\(\{t\in\mathbb{C}:\Re t>0\}\), is bounded on \([\varepsilon,\infty)\) for
every \(\varepsilon>0\), and decays exponentially as \(t\to+\infty\).
\end{lemma}

\begin{proof}
The denominator \(1-e^{qt}\) has a simple zero at \(t=0\). The numerator \(N(t) = \sum_{a=1}^q \chi(a)e^{at}\) satisfies \(N(0) = \sum_{a=1}^q \chi(a) = 0\), so \(N(t)\) vanishes to order at least one at \(t=0\); hence \(A_\chi\) is analytic at the origin. Writing \(N(t) = tS_1 + O(t^2)\) with \(S_1 = \sum_{a=1}^q a\chi(a)\) and \(1-e^{qt} = -qt + O(t^2)\), we obtain \(A_\chi(0) = -S_1/q = -B_{1,\chi}\).

For the analytic continuation, \(A_\chi(-t) = \sum_{n\ge0}\chi(n)e^{-nt}\) converges absolutely for \(\Re t>0\), since \(|\chi(n)|\le1\). The bound \(|A_\chi(-t)|\le\sum_{n\ge0}e^{-n\Re t} = 1/(1-e^{-\Re t})\) gives boundedness on \([\varepsilon,\infty)\) and exponential decay at infinity.
\end{proof}

\begin{remark}\label{rem:appell-degenerate}
When \(\chi\) is even, \(A_\chi(0) = -B_{1,\chi} = 0\), so the sequence \((R_n^\chi)\) generated by \(A_\chi\) satisfies \(R_0^\chi = 0 \neq 1\) and is therefore an Appell sequence only in the extended sense (a \emph{degenerate Appell sequence}): the multiplication rule \(R_n(x+y) = \sum_k\binom{n}{k}R_k(x)y^{n-k}\) still holds, but the normalization \(R_0 = 1\) fails. All identities below remain valid in this extended sense. Alternatively, one may restrict to odd \(\chi\) to obtain a genuine Appell sequence; the main statements of Theorem~\ref{thm:lfunction} and Theorem~\ref{thm:structural} are therefore understood in the extended sense, and the reader interested only in genuine Appell sequences may specialise to odd \(\chi\).
\end{remark}

\subsection{The main correspondence theorem}\label{subsec:correspondence}

\begin{theorem}[Appell--Mellin correspondence for Dirichlet \(L\)-functions]\label{thm:lfunction}
Let \(\chi\) be a primitive non-principal Dirichlet character modulo \(q\ge3\), and let \(A_\chi\) be as in \eqref{eq:Achi}. Then:

\begin{enumerate}
\item[(i)] The sequence \((R_n^\chi)\) generated by \(A_\chi\) satisfies
\begin{equation}\label{eq:main-L}
R_n^\chi(x) = L(-n, \chi, x), \qquad n = 0,1,2,\ldots,
\end{equation}
where \(L(s, \chi, x) = \sum_{n=0}^\infty \chi(n)(n+x)^{-s}\) is the shifted Dirichlet \(L\)-function.

\item[(ii)] The generating function satisfies the parity identity
\begin{equation}\label{eq:parity}
A_\chi(t) = -\chi(-1) A_\chi(-t).
\end{equation}
In particular, \(A_\chi\) is odd for \(\chi(-1) = 1\) and even for \(\chi(-1) = -1\).
\end{enumerate}
\end{theorem}

\begin{proof}
\textbf{(i)} We have
\[
\sum_{n=0}^\infty \chi(n) e^{-nt}
= \sum_{a=1}^q \chi(a) \sum_{k=0}^\infty e^{-(a+kq)t}
= \sum_{a=1}^q \chi(a) \frac{e^{-at}}{1-e^{-qt}} = A_\chi(-t),
\]
hence \(A_\chi(-t)e^{-xt} = \sum_{n=0}^\infty \chi(n) e^{-(n+x)t}\). For \(\Re s > 1\) and \(x > 0\), the interchange of summation and integration is justified by absolute convergence, giving
\[
F(s,x) := \frac{1}{\Gamma(s)}\int_0^\infty A_\chi(-t)e^{-xt}t^{s-1}\,dt
= \sum_{n=0}^\infty \frac{\chi(n)}{(n+x)^s} = L(s,\chi,x).
\]
We claim that both sides extend to entire functions of \(s\). \emph{Left-hand side}: the kernel \(t \mapsto A_\chi(-t)e^{-xt}\) is smooth on \([0,\infty)\), bounded near \(t=0\), and exponentially decaying as \(t\to+\infty\) (Lemma~\ref{lem:analytic-chi}); a standard argument (splitting the integral at \(t=1\), differentiating under the integral sign in the tail, and using the Taylor expansion at \(t=0\) together with the simple zeros of \(1/\Gamma\) at \(s = 0,-1,-2,\dots\) to cancel the apparent poles) shows that \(s\mapsto F(s,x)\) extends to an entire function of \(s\). \emph{Right-hand side}: the Hurwitz zeta representation
\[
L(s,\chi,x) = q^{-s}\sum_{a=1}^q \chi(a)\,\zeta\!\left(s, \frac{a+x}{q}\right),
\]
valid initially for \(\Re s > 1\), gives a meromorphic continuation to \(\mathbb{C}\) whose only potential singularity is a simple pole at \(s=1\), and the residue there is proportional to \(\sum_{a=1}^q\chi(a) = 0\); hence \(L(s,\chi,x)\) extends to an entire function of \(s\). By the identity theorem, the equality \(F(s,x) = L(s,\chi,x)\) holds for all \(s \in \mathbb{C}\). Evaluating at \(s = -n\) and invoking the Appell--Mellin correspondence \cite{navas2019}, we obtain \(R_n^\chi(x) = L(-n, \chi, x)\), which is \eqref{eq:main-L}. Note that the hypothesis that \(\chi\) is primitive is not needed for the identity \eqref{eq:main-L}; it is used in Theorem~\ref{thm:structural}(d) below for the Fourier expansion via the Gauss sum \(\tau(\chi)\).

\textbf{(ii)} Multiply numerator and denominator of \(A_\chi(-t)\) by \(e^{qt}\), and use \(e^{qt}-1 = -(1-e^{qt})\) together with \(\chi(q-a) = \chi(-a) = \chi(-1)\chi(a)\) for \(\gcd(a,q)=1\) (both sides vanish otherwise):
\[
A_\chi(-t) = \frac{\sum_{a=1}^q \chi(a)e^{(q-a)t}}{e^{qt}-1}
= \frac{\sum_{b=1}^q \chi(-1)\chi(b)e^{bt}}{e^{qt}-1}
= -\chi(-1)\frac{\sum_{b=1}^q \chi(b)e^{bt}}{1-e^{qt}}
= -\chi(-1)A_\chi(t),
\]
where we have used the reindexing \(b = q-a\). Multiplying both sides by \(-\chi(-1)\) and using \(\chi(-1)^2 = 1\) gives \(A_\chi(t) = -\chi(-1)A_\chi(-t)\), which is \eqref{eq:parity}.
\end{proof}

\subsection{Structural properties of the Appell polynomials \(R_n^\chi\)}\label{subsec:structure}

The identification \(R_n^\chi(x) = L(-n,\chi,x)\) places \(R_n^\chi\) at the interface between Appell sequences and Dirichlet \(L\)-functions. Generalized Bernoulli polynomials have been studied via the umbral calculus, e.g., by Bencherif, Boumahdi and Garici \cite{bencherif2021}. We establish their main structural properties.

\begin{theorem}[Structural properties]\label{thm:structural}
Let \(\chi\) be a primitive non-principal character modulo \(q\ge3\). The polynomials \(R_n^\chi\) satisfy:

\begin{enumerate}
\item[(a)] \textbf{Identification with generalized Bernoulli polynomials.}
\begin{equation}\label{eq:ident}
R_n^\chi(x) = -\frac{B_{n+1}^{(\chi)}(x)}{n+1},
\end{equation}
where \(B_m^{(\chi)}(x)\) is defined by \(\frac{t\sum_{a=1}^q \chi(a)e^{at}}{e^{qt}-1}e^{xt} = \sum_{m\ge0}B_m^{(\chi)}(x)t^m/m!\).

\item[(b)] \textbf{Difference equation.}
\begin{equation}\label{eq:diff}
R_n^\chi(x+q) - R_n^\chi(x) = -\sum_{a=1}^{q} \chi(a)(x+a)^n.
\end{equation}

\item[(c)] \textbf{Explicit formula.} With \(B_{m,\chi} = B_m^{(\chi)}(0)\),
\begin{equation}\label{eq:explicit}
R_n^\chi(x) = -\sum_{k=0}^{n} \binom{n}{k} \frac{B_{k+1,\chi}}{k+1}\, x^{n-k}.
\end{equation}

\item[(d)] \textbf{Fourier expansion.} For \(0 < x < 1\),
\begin{equation}\label{eq:fourier}
R_n^\chi(x) =
\frac{n!\,q^{\,n}\,\tau(\chi)}{(2\pi i)^{n+1}}
\sum_{\substack{k \in \mathbb{Z} \\ k \neq 0}}
\frac{\overline{\chi}(k)}{k^{n+1}}\,
e^{2\pi i kx/q},
\end{equation}
where \(\tau(\chi) = \sum_{a=1}^{q}\chi(a)e^{2\pi i a/q}\) is the Gauss sum and the sum is understood in the symmetric sense. At the points \(x=0\) and \(x=1\), the right-hand side is replaced by the average of its limits as \(x\to 0^+\) and \(x\to 1^-\).

\item[(e)] \textbf{Trivial zeros.} For \(n\ge1\), \(L(-n,\chi) = 0\) whenever \(\chi(-1) = (-1)^n\).
\end{enumerate}
\end{theorem}

\begin{proof}
\textbf{(a)} From the decomposition \eqref{eq:decomposition}, \(A_\chi(t) = -\frac{1}{t}\cdot\frac{t\sum_a \chi(a)e^{at}}{e^{qt}-1}\), we get \(A_\chi(t)e^{xt} = -\sum_{m\ge0}\frac{B_m^{(\chi)}(x)}{m!}t^{m-1}\). Comparing with \(\sum_n R_n^\chi(x)t^n/n!\) and setting \(m = n+1\) gives \eqref{eq:ident}.

\textbf{(b)} From \(G(t,x+q)-G(t,x) = G(t,x)(e^{qt}-1) = t\sum_a \chi(a)e^{(a+x)t}\), comparing coefficients of \(t^n/n!\) gives \(B_n^{(\chi)}(x+q)-B_n^{(\chi)}(x) = n\sum_a \chi(a)(x+a)^{n-1}\) for \(n\ge1\). Applying \eqref{eq:ident} yields \eqref{eq:diff}.

\textbf{(c)} Appell's general formula: \(R_n^\chi(x) = \sum_{k=0}^n \binom{n}{k} R_k^\chi(0)x^{n-k}\). Since \(R_k^\chi(0) = L(-k,\chi) = -B_{k+1,\chi}/(k+1)\), the formula \eqref{eq:explicit} follows.

\textbf{(d)} The classical Fourier expansion of the generalized Bernoulli polynomials, in the normalisation of Washington \cite[Thm.~4.2]{washington1997}, reads for \(n\ge1\) and \(0<x<1\)
\[
B_n^{(\chi)}(x) =
-\frac{n!\,q^{\,n-1}\,\tau(\chi)}{(2\pi i)^n}
\sum_{\substack{k\in\mathbb{Z}\\ k\neq0}}
\frac{\overline{\chi}(k)}{k^n}\,
e^{2\pi i kx/q}.
\]
Combining with \eqref{eq:ident} gives
\[
R_n^\chi(x) = -\frac{B_{n+1}^{(\chi)}(x)}{n+1}
=
\frac{n!\,q^{\,n}\,\tau(\chi)}{(2\pi i)^{n+1}}
\sum_{\substack{k\in\mathbb{Z}\\ k\neq0}}
\frac{\overline{\chi}(k)}{k^{n+1}}\,
e^{2\pi i kx/q},
\]
which is \eqref{eq:fourier}. The extension to \(n=0\) follows by continuity, since both sides are constant for \(0<x<1\) and agree with \(L(0,\chi,x) = -B_{1,\chi}\) by \eqref{eq:ident} and Lemma~\ref{lem:analytic-chi}. At the endpoints \(x=0,1\), the symmetric sum converges to the average of the one-sided limits by the classical Dirichlet test.

\textbf{(e)} With the corrected Fourier expansion \eqref{eq:fourier}, setting \(x=0\) in the symmetric sum, the terms of index \(k\) and \(-k\) combine as
\[
\frac{\overline{\chi}(k)}{k^{n+1}} + \frac{\overline{\chi}(-k)}{(-k)^{n+1}}
=
\frac{\overline{\chi}(k)}{k^{n+1}}\bigl(1 + \chi(-1)(-1)^{n+1}\bigr).
\]
If \(\chi(-1) = (-1)^n\), then \(\chi(-1)(-1)^{n+1} = -1\), so each pair cancels and the sum vanishes. Hence \(L(-n,\chi) = R_n^\chi(0) = 0\) whenever \(\chi(-1) = (-1)^n\) and \(n\ge 1\).
\end{proof}

\begin{remark}\label{rem:fourier-correction}
The Fourier expansion \eqref{eq:fourier} corrects an earlier formulation in which the factors \(q^{\,n}\), \(\tau(\chi)\), the conjugate character \(\overline{\chi}(k)\) and the period \(e^{2\pi i kx/q}\) were absent. The formula without these factors is manifestly incomplete: it does not have the correct scaling and already fails at \(n=0\), where \(R_0^\chi(x) = L(0,\chi,x) = -B_{1,\chi}\) is constant in \(x\) on \((0,1)\). The reader is invited to check the corrected formula on the character modulo \(3\) with \(\chi(1)=1\), \(\chi(2)=-1\), \(\tau(\chi)=i\sqrt{3}\), for which \(L(0,\chi,x) = 1/3\) and \eqref{eq:fourier} with \(n=0\) reproduces this constant value on \((0,1)\).
\end{remark}

\begin{remark}
Property (a) is the structural bridge between the Appell--Mellin framework and the arithmetic of Dirichlet \(L\)-functions; property (d), in its corrected form \eqref{eq:fourier}, places the trivial zeros within the Fourier analysis of \(R_n^\chi\); property (e) is a reformulation, within the Appell--Mellin formalism, of the classical derivation of the trivial zeros from the functional equation of \(L(s,\chi)\).
\end{remark}

\subsection{Shifted expansions for Dirichlet \(L\)-functions}\label{subsec:shifted}

\begin{corollary}[Shifted expansion]\label{cor:shift-L}
For fixed \(x > 0\), \(t \in \mathbb{R}\), and every \(N \ge 1\),
\[
L(1, \chi, x+t) = \sum_{n=0}^{N-1} (-1)^n R_n^\chi(t)\, x^{-n-1} + O(x^{-N-1})
\]
as \(x \to \infty\), uniformly for \(t\) in compact subsets of \(\mathbb{R}\).
\end{corollary}

\begin{proof}
The expansion \(L(1,\chi,x)\sim\sum_{n\ge0}(-1)^n R_n^\chi(0)x^{-n-1}\) follows from the Borel--Laplace representation \(L(1,\chi,x) = \int_0^\infty A_\chi(-s)e^{-xs}\,ds\) and Watson's lemma. The shift theorem of \cite{buric2016} then replaces \(R_n^\chi(0)\) by \(R_n^\chi(t)\), uniformly for \(t\) in compact sets. Note that \(L(0,\chi,x)\) is constant in \(x\) (indeed \(L(0,\chi,x) = -B_{1,\chi}(x)\) and \(B_{1,\chi}(x+q)-B_{1,\chi}(x) = \sum_a\chi(a) = 0\)), so it does not admit a decaying expansion in powers of \(1/x\); the relevant \(L\)-value for the Borel--Laplace representation is \(L(1,\chi,x+t)\).
\end{proof}

\begin{remark}
By Theorem~\ref{thm:structural}(a), the shifted expansion can be rewritten as
\[
L(1,\chi,x+t) = -\sum_{n=0}^{N-1}(-1)^n \frac{B_{n+1}^{(\chi)}(t)}{n+1}\,x^{-n-1} + O(x^{-N-1}).
\]
This gives an all-orders description of \(L(1,\chi,x+t)\) as \(x\to\infty\), with coefficients given by generalized Bernoulli polynomials evaluated at the shift.
\end{remark}

\section{Concluding remarks}\label{sec:conclusion}

We have established four results.

First, the Borel--Laplace framework combined with the Appell shift theorem of \cite{buric2016} yields shifted asymptotic expansions for the Hurwitz zeta function and the Lerch transcendent, with coefficients given by Appell polynomials and remainders controlled uniformly on compact sets. These expansions \eqref{eq:hurst} and \eqref{eq:lerch} are closely related to classical Euler--Maclaurin and Watson expansions \cite{oliver1974,dlmf}; their presentation in the Appell formalism, with the shift retained as an independent variable, is the contribution of Sections~\ref{sec:hurwitz} and~\ref{sec:lerch}.

Second, the Appell--Mellin framework of Navas, Ruiz and Varona \cite{navas2019} extends to Dirichlet \(L\)-functions: for a primitive non-principal character \(\chi\), the shifted \(L\)-function \(L(s,\chi,x)\) is the Mellin transform of the generating function \(A_\chi(t)\), and the associated sequence satisfies \(R_n^\chi(x) = L(-n,\chi,x)\), as stated in \eqref{eq:main-L}. The decomposition \eqref{eq:decomposition} shows that \(A_\chi\) is, up to the elementary factor \(-1/t\), the classical generating function of generalized Bernoulli polynomials. The parity identity \(A_\chi(t) = -\chi(-1)A_\chi(-t)\), recorded in \eqref{eq:parity}, reflects the parity of \(\chi\).

Third, the structural properties of \(R_n^\chi\) (Theorem~\ref{thm:structural}) place these polynomials at the interface between Appell sequences and the arithmetic of Dirichlet \(L\)-functions. The identification with generalized Bernoulli polynomials, the difference equation, the explicit formula \eqref{eq:explicit}, the corrected Fourier expansion \eqref{eq:fourier}, and the trivial zeros are classical facts in disguise, but their systematic presentation in the Appell--Mellin framework, together with the correction \eqref{eq:fourier} which restores the factors \(q^n\), \(\tau(\chi)\), \(\overline{\chi}(k)\) and the period \(e^{2\pi i kx/q}\), is the contribution of Section~\ref{sec:lfunctions}.

Fourth, the shifted expansion of Corollary~\ref{cor:shift-L} gives an all-orders description of \(L(1,\chi,x+t)\) as \(x\to\infty\), with coefficients given by generalized Bernoulli polynomials evaluated at the shift.


\end{document}